\documentclass[12pt]{amsart}

\usepackage[T1]{fontenc}
\usepackage[utf8]{inputenc}
\usepackage{amsmath,amsthm,amssymb}
\usepackage{graphicx}
\usepackage{enumerate}
\usepackage{color}

\usepackage{dsfont}
\usepackage{hyperref}
\hypersetup{
    colorlinks=true,
    linkcolor=blue,
    citecolor=red,
    urlcolor=blue,
    pdfborder={0 0 0}
}
\usepackage{fullpage}

\newcommand{\F}{\mathcal F}
\newcommand{\E}{\mathds E}
\newcommand{\R}{\mathbb R}
\newcommand{\N}{\mathbb N}

\newcommand{\com }{\mathbb C}
\renewcommand{\P}{\mathds P}

\newcommand{\Ll}{\mathcal{L}}
\newcommand{\Var}{\mathrm{Var}}

\renewcommand{\d}{\delta}
\renewcommand{\a}{\alpha}
\newcommand{\eps}{\varepsilon}
\newcommand{\g }{\gamma}

\newcommand{\1}{{\mathds{1}}}
\newcommand{\8}{\infty}
\newtheorem{lemma}{Lemma}

\newtheorem{theorem}{Theorem}

\newtheorem{remark}{Remark}

\begin{document}
	\title[Absolute continuity of the martingale limit in BPRE]{Absolute continuity of the martingale limit in branching processes in random environment}
	\author[]{Ewa Damek, Nina Gantert, Konrad Kolesko}

  \address{Ewa Damek, Mathematical Institute, University of
Wroc{\l}aw, Plac Grunwaldzki 2/4, 50-384 Wroc{\l}aw, Poland }
	\email{edamek@math.uni.wroc.pl}

  \address{Nina Gantert, Mathematical Institute, Tech\-nische Universit\"at M\"unchen, 
  Boltzmannstr.~3, 85748 Garching, Germany}
	\email{gantert@ma.tum.de}

  \address{Konrad Kolesko, Mathematical Institute, University of
Wroc{\l}aw, Plac Grunwaldzki 2/4, 50-384 Wroc{\l}aw, Poland \newline 
Institut f\"ur Mathematik, Universit\"at Innsbruck, Technikerstra\ss{}e 13, 6060 Innsbruck, Austria. }
	\email{kolesko@math.uni.wroc.pl}

\maketitle

	{\bf Abstract.}
We consider a supercritical branching process $Z_n$ in a stationary and ergodic random environment $\xi =(\xi_n)_{n\ge0}$. 
Due to the martingale convergence theorem, it is known that the normalized population size $W_n=Z_n\slash (\E (Z_n|\xi ))$ converges almost surely to a random variable $W$. We prove that if $W$ is not concentrated at $0$ or $1$ then for almost every environment $\xi$ the law of $W$ conditioned on the environment $\xi $ is absolutely continuous with a possible atom at $0$. The result generalizes considerably the main result  of \cite{kaplan:1974}, and of course it covers the well-known case of the martingale limit of a Galton-Watson process. Our proof combines analytical arguments with the recursive description of $W$. 



\section{Introduction and statement of the main result}

There has been a lot of interest in asymptotic properties of $W$ e.g convergence rates of $W-W_n$ as well as limit theorems for $Z_n$ and large deviations principles. Positive and negative, annealed and quenched,  moments of $W$ were studied. Most of that was done for the i.i.d environment, because then properties of the so-called ``associated random walks'' could be applied, but some results hold also in a stationary and ergodic environment. For a sample of results see \cite{bansaye:berestycki:2009, bansaye:boinghoff:2014,bansaye:boinghoff:2017, grama:liu:miqueu:2017,huang:liu:2014,li:liu:gao:wang:2014} and references therein.

However, except of \cite{kaplan:1974} the local regularity of the law of $W$ has not been studied. Due to the basic equation \eqref{Wequation} satisfied by $W$ it is closely related to the local regularity for stationary solutions to affine type equations \eqref{smooth} which is partly our motivation as it is explained in the end of the introduction. We stress that our arguments are valid for stationary and ergodic environment.   	

All our random variables are defined on a probability space $(\Omega,\F,\P)$.
 	
Let $\Delta$ be the space of probability measures on $\N _0=\{0,1,2,...\}$ - the set of possible offspring distributions. Let $\xi =(\xi_n)_{n\ge 0}$ be a stationary and ergodic process taking values in 
$\Delta $.  The sequence $(\xi_n)_{n\ge0}$ is called a "random environment" or "environment sequence".

The process $(Z_n: n\geq 0)$ with values in
$\N _0$ is called a branching process in random environment $\xi$ if $Z_0$ is independent of $\xi$ and it satisfies
\begin{equation}\label{defbran}
\Ll (Z_n|\xi , Z_0,\dots Z_{n-1})=\xi _{n-1}^{*Z_{n-1}}\quad \mbox{a.s.}
\end{equation} 
where $\xi _{n-1} ^{*k}$ is the $k$ fold convolution.
For an environment sequence $\xi$ we denote
$$f_{\xi _n}(s)=\sum _{k= 0}^{\8}s^k\xi _n (\{ k\}), \quad s\in \com,\ |s|\leq 1,$$
the sequence of probability generating functions associated with  $\xi$ and 
$$m_n=m(\xi _n)=f'_{\xi _n}(1)=\sum _{k= 0}^{\8}k\xi _n (\{ k\}),$$ 
the sequence of the means. 

 By $\P_{\xi}$ we denote the measure $\P$ conditioned on the environment $\xi$. The corresponding mean and variance are denoted by $\E_{\xi}$ and $\Var_{\xi} $ i.e. for any random variable $X$ we have  $\E_{\xi}[X]=\E[X|\xi]$ and $\Var_{\xi}(X)=\E\big[(X-\E_{\xi}X)^2|\xi\big]$. For any random variable $X$ also introduce the conditional law $\mathcal L_{\xi}(X)$ by $\mathcal L_{\xi}(X)(A)=\P(X\in A|\xi)$, for any measurable set $A$.

In this notation we may write
\begin{equation}\label{Zgenerating}
F_n(s, \xi) = \E_\xi[s^{Z_n}|Z_0,\dots , Z_{n-1}]=f_{\xi _{n-1}}(s)^{Z_{n-1}}\quad \mbox{a.s.}\end{equation}
Conditioned on the past and on the environment sequence, $Z_n$ may be viewed as the sum of $Z_{n-1}$ independent and identically distributed random variables $Y_{n-1,i}$, each having $f_{\xi _{n-1}}(s)$ as its probability generating function. Then the process $\{ Z_n\} _{n=0}^{\8 }$ conditioned on the environment $\xi $ is called a branching process in varying environment.
Iterating \eqref{Zgenerating} we obtain 
\begin{equation}
\E_\xi [s^{Z_n}|Z_0=m]=\left (f_{\xi _0}(f_{\xi _1}(...(f_{\xi _{n-1}}(s))))\right )^m= F_{n}(s, \xi)^m\end{equation}     
Let
$$
q (\xi )=\P_\xi\Big(\lim _{n\to \8} Z_n=0\Big|Z_0=1\Big)
$$
be the extinction probability of the process $Z_n$. Since $\xi $ is ergodic,
$\P (q(\xi )<1)$ equals $0$ or $1$ a.s.. We assume that the random variable $\log m_0$
is integrable. If $\E \log m_0\leq 0$ then it is easy to see that $\P (q(\xi )=1)=1$, see also \cite{tanny:1977}, unless $\xi _0=\delta _1$ a.s.
On the other hand, if
\begin{equation}\label{meanpos}
0< \mu: =\E \log m_0 <\8 
\end{equation}
then a sufficient condition for $\P (q(\xi )<1)=1$ is 
\begin{equation}\label{nonextin}
 \E |\log (1-f_{\xi _0}(0))|<\8 , 
\end{equation}
see \cite{athreya:karlin:1971}. Moreover, it was proved in \cite{smith:wilkinson:1969} that for  i.i.d. sequences  $(\xi _n)_{n \geq 0}$  the condition \eqref{nonextin} is also necessary for  $\P (q(\xi )<1)=1$ to hold.
See also \cite{athreya:karlin:1971} and \cite{kersting:vatutin:2017}. 
Let $$
M_n=\E_{\xi}[Z_n]=m_0\cdot \ldots \cdot m_{n-1}$${ and }  $$W_n=\frac{Z_n}{M_n}. $$
Then $(W_n)_{ n \ge 0}$ is a nonnegative martingale under $\P_{\xi}$. Therefore,
$$
\lim _{n\to \8 } W_{n }=W$$
exists $\P_{\xi}$--almost surely.
Our main result is the following description of the law of $W$ under $\P_{\xi}$.
\begin{theorem}\label{main}
	Suppose that the environment sequence   $\xi $ is stationary and ergodic and \eqref{meanpos} holds. Let $\mathcal{L} _{\xi}$ be the law of $W$ under $\P _{\xi}$. 
	Then exactly one of the following three cases occurs:
	\begin{enumerate}[(i)]
		\item $\mathcal{L} _{\xi}=\delta _0$ a.s.
	\item $\mathcal{L} _{\xi}=\delta _1$ a.s.
			\item $q(\xi) < 1$ and $\mathcal L_{\xi}(W)=q(\xi)\d _0+\nu _{\xi }$, a.s where $\nu _{\xi}$ is absolutely continuous with respect to the Lebesgue measure. 
	\end{enumerate} 
\end{theorem}

The following two recursive formulas will be crucial for our proof.
The definition of the process $Z_n$ yields that $W$ satisfies the relation
\begin{equation}\label{Wequation}
W=\frac{1}{m_0}\sum _{j=1}^{Z_{1}}W_{j},
\end{equation}
where under $\P_{\xi}$, the random variables $W_{j}$ are independent of each other  and independent of $Z_1$ with distribution  $\P_{\xi}(W_j\in\cdot)=\P_{T\xi}(W\in\cdot)$.
We write
$$
\psi (t,\xi )=\E_{\xi}  [e^{itW}]$$
for the  conditional characteristic function of $W$.
Then by the recursive relation \eqref{Wequation} we obtain
\begin{align}\label{recpsi}
\psi (t,\xi )&=
f_{\xi _0}(\psi (t\slash m_0^{-1},T\xi))\nonumber \\
&=f_{\xi _0}\circ \dots\circ f_{\xi _{n-1}}(\psi (t\slash M_n,T^n\xi))\nonumber \\
&=F_n(\psi (t\slash M_n,T^n\xi)),
\end{align}
where $F_n$ is the probability generating function of $Z_n$, see \eqref{Zgenerating}.

The question about local regularity of $\mathcal{L} _{\xi}(W)$ fits very well into a number of similar problems being investigated recently. For the Galton Watson process the $W_j$'s have the same law as $W$ and so then \eqref{Wequation} is an example of the so-called smoothing equation. By the latter we mean 
\begin{equation}\label{smooth}
Y=\sum _{j\geq 1}T_jY_j+C,
\end{equation}
where the equality is meant in law, $(C,T_1,T_2,....)$ is a given sequence of real or complex random variables and $Y_1,Y_2,...$ are independent copies of the variable $Y$ and independent of $(C,T_1,T_2,....)$. 
Let $N$ be a random number of $T_j$'s that are not zero. As long as $\E N>1$ the transform 
$$
S(\mu )=\mbox{Law of }(\sum _{j\geq 1}T_jY_j+C),$$
where $\mu $ is the law of $Y_1$, improves local regularity of the measure, and so it is expected that the fixed points of $S$ are absolutely continuous even when the $T_j$'s and $C$ are discrete. This is indeed the case, see \cite{damek:mentemeier:2018}, \cite{Leckey:2018} and \cite{Liu:2001}.

However, in the case of a random environment, the equation \eqref{Wequation} is not exactly of the form in \eqref{smooth} and so a different approach had to be elaborated.

If $N=1$ a.s., \eqref{smooth} becomes
\begin{equation}\label{affine}
Y=TY+C
\end{equation}
and absolute continuity of the solution is much harder to prove if $(T,C)$ does not possess a priori any regularity, as
for instance in the case of Bernoulli convolutions $T$ is concentrated at $\lambda $,
for some $0<\lambda <1$ and $C$ is a Bernoulli random variable, i.e. $C$ takes the values $+1, -1$ each with probability $1\slash 2$. If $0<\lambda <1/2$ then the law $\nu _{\lambda }$ of $Y=\lambda Y+C$ is continuous but singular with respect to Lebesgue measure and if $\lambda =1/2$ then $\nu _{\lambda }$ is the uniform distribution on $[-2,2]$. 
However, when $1\slash 2<\lambda <1$, $\nu _{\lambda }$ is absolutely continuous for almost every such $\lambda $, \cite{solomyak:1995}, or even better: it is absolutely continuous outside of a subset of $\lambda \in (1\slash 2,1) $  of Hausdorff dimension $0$,  \cite{shmerkin:2014}. Moreover, if particular $\lambda$'s are considered, absolute continuity of $\nu _{\lambda}$ depends on delicate algebraic properties of $\lambda$, see \cite{varju} for an overview of the recent developments on Bernoulli convolutions.   

When we go beyond Bernoulli convolutions there is no general theory about regularity of $\nu $.  Further examples of singular $(T,C)$ that give rise to absolutely continuous solutions as well as to singular ones are given in \cite{brieussel:tanaka:2015},  
\cite{peres:solomyak:1996}, \cite{peres:solomyak:1998}, \cite{pratsiovytyi:khvorostina:2013}, \cite{swiatkowski:2017}.

In order to prove Theorem \ref{main}, we will need some additional statements provided in the next section.    

\section{Further results} 
 In general, for a supercritical BPRE, $W$ may vanish almost surely and conditions for that to happen are well known. Notice that due to \eqref{Wequation}, the sets $\{ \xi : \mathcal{L} _{\xi }=\d _0\}$ and $\{ \xi : \mathcal{L} _{T\xi }=\d _0\}$ differ by a set of measure zero so $\P ( \xi : \mathcal{L} _{\xi }=\d _0)\in \{0,1\}$. If  $\P ( \xi : \mathcal{L} _{\xi }=\d _0) =0$, i.e. if $W$ is not identically zero, let $z(\xi)=\P _{\xi }(W=0)<1$. In fact, as explained below, it is known, that $z(\xi)=q(\xi)$
 but we will not need this information for our proof of Theorem \ref{Kap}.
 We say that a measure is {\it degenerate} if it is concentrated at a point.
\begin{theorem}\label{Kap}
Suppose that the environment sequence $\xi $ is stationary and ergodic, \eqref{meanpos} holds,
$\P ( \xi : \mathcal{L} _{\xi }=\d _0)= 0$ and $\P (\xi: \xi _0\ \mbox{not degenerate})>0$. 
Then 
$$ \mathcal L_{\xi}(W)=z(\xi)\d _0+\nu _{\xi }\ \mbox{a.s.,} $$
where  $\nu _{\xi}$ is absolutely continuous with respect to Lebesgue measure.
\end{theorem}
\begin{remark}
	Theorem \ref{main} follows directly from Theorem \ref{Kap}. Indeed, if $\P ( \xi : \mathcal{L} _{\xi }=\d _0)=1$ then (i) in Theorem \ref{main} holds. If $\mu >0$ (recall \eqref{meanpos}) and $\P (\xi: \xi _0\ \mbox{degenerate})=1$ then $W_n$ is concentrated at $1$ for every $n$, hence the same is true for $W$.
Moreover,
if $W$ is not identically zero then $z(\xi)=q(\xi)$, see \cite{tanny:1978} and \cite{tanny:1988}.
Let us provide a short argument. If \eqref{meanpos} holds,
then there is a sequence of random variables $c_n(\xi )$ such that
\begin{equation*}
\lim _{n\to \8}c_n^{-1}Z_n=U \,  \rm{ a.s.} \end{equation*}
and
\begin{equation*}
\P _{\xi }(U=0)=q(\xi ),\quad \P _{\xi }(U<\infty )=1,
\end{equation*}
see \cite{tanny:1978}, Theorem 1. But
$$
\frac{Z_n}{M_n}=\frac{Z_n}{c_n}\frac{c_n}{M_n}.$$
Since 
\begin{equation}
\label{UandW}
\P _{\xi }(U=0)\leq \P _{\xi }(W=0)<1,
\end{equation}
$$
\lim _{n\to \8}\frac{c_n}{M_n}=L(\xi )\ a.s.\ \mbox{and}\ 0<L(\xi)<\8 . $$
More precisely, note that since $\frac{c_n}{M_n}$ is constant under $\P_\xi$, it suffices to show that
its limit is strictly positive with positive $\P_\xi$-probability, and this is true due to \eqref{UandW}.
Hence
$$
W=L(\xi )U$$  
and $\P _{\xi }(W=0)=\P _{\xi }(U=0)=q(\xi )$.  
\end{remark}
\begin{remark}
The question when $W$ is not identically zero is well-studied.
For a stationary and ergodic environment a sufficient condition was given in \cite{athreya:karlin:1971a}.
\begin{theorem}
(see \cite{athreya:karlin:1971a}) Let $Z_0=1$. Suppose that \eqref{meanpos} is satisfied and
\begin{equation}\label{kestenstigum}
\E [m_0^{-1}Z_1\log ^+ Z_1] <\8.\end{equation}
Then
	\begin{equation}\label{notzero}
	W=\lim _{n\to \8}\frac{Z_n}{M_n}\ 
	 \mbox{is not identically zero.}\end{equation}
	Furthermore,
	\begin{equation}\label{q}
	\P _{\xi}(W=0)=q(\xi ) \quad \mbox{a.s.}
	\end{equation}
 and
\begin{equation}\label{expec}
\E _{\xi }W=1\quad \mbox{a.s.}
\end{equation} \end{theorem} 

Moreover, it was proved in \cite{tanny:1988}  that if  $(\xi_n) $ is an i.i.d. sequence then condition \eqref{kestenstigum} is in fact equivalent to \eqref{notzero}. Another proof for  i.i.d environements $(\xi _n)$ is contained in \cite{kersting:vatutin:2017}. For i.i.d environments, assuming \eqref{meanpos}, \eqref{kestenstigum} and \eqref{expec} are equivalent.  
In general, when the sequence $(\xi_n)$ is assumed to be only  stationary and ergodic  \eqref{kestenstigum} is not necessary for $W$ to be not identically zero \cite{tanny:1988}. In this case the necessary condition is 
 $$
 \sum _{n=0}^{\8}\frac{1}{m_n}\bigg(\sum _{k\geq M_{n+1}}k\xi _n(k)\bigg )<\8 \ a.s.$$
 The sufficient condition is only a little bit stronger (see Theorem 1, \cite{tanny:1988}). 
Under this sufficient condition, \eqref{q} and \eqref{expec} hold. 
\end{remark}

\begin{remark}
Theorem \ref{Kap} generalizes considerably Theorem 1 in \cite{kaplan:1974} but, what is more important, Kaplan's proof contains essential gaps that concern the integrability of $|\psi '(\cdot , \xi)|$. We don't think that they are easily reparable within his approach and instead we suggest our proof which is contained in Theorem \ref{integrability} below.  However, the idea to show the integrability of $|\psi '(\cdot , \xi)|$ is borrowed from \cite{kaplan:1974}.
\end{remark}
 In order to prove Theorem \ref{Kap} we use the following analytical result. 
\begin{lemma}\label{distr}
Let $\nu $ be a probability measure on $(\R, {\mathcal B})$ with finite first moment and let $\psi $ be its characteristic function. If $|\psi '|$ is integrable then $\nu = c\d _0+\nu _{\rm abs}$ where $\nu _{\rm abs}$ is absolutely continuous with respect to the Lebesgue measure.
\end{lemma}
\begin{proof}
	$\partial _t\psi (t) \ dt$ defines a tempered distribution, see \cite{rudin}, part 2. Moreover, its Fourier inverse satisfies
	\begin{equation*}
	\F ^{-1}(\partial _t\psi (t)\ dt)=\F ^{-1}(\partial _t\psi (t))\ dx=: f(x)\ dx,
	\end{equation*}
	where $f$ is a complex valued function vanishing at infinity. In the above formula the first $\F ^{-1}$ means the inverse Fourier transform of a tempered distribution and the second $\F ^{-1}$ the inverse Fourier transform of an integrable function. On the other hand
	\begin{equation*}
	\F ^{-1}(\partial _t\psi (t)\ dt)=-ix\F ^{-1}(\psi (t)\ dt)=-ix\nu ,
	\end{equation*}
	as tempered distributions. Hence
	$$
	-ix\nu = f(x)\ dx .$$
	This shows that $\nu \1 _{\R \setminus \{ 0\}}$ has density given by $-ix^{-1}f(x)$ and the conclusion follows.
\end{proof}
The key step in the proof of Theorem  \ref{Kap} is the following theorem.
\begin{theorem}\label{integrability}
Suppose that $\xi $ is stationary and ergodic, \eqref{meanpos} holds 
and for a.e. $\xi $, 
\begin{equation}\label{defrho}
\rho (\xi )=\sup _{|x|\geq 1}|\psi (x,\xi )|<1.
\end{equation}
Then  
for a.e $\xi $, $\int _{\R }|\psi '(t,\xi )|\ dt <\8 $. 
\end{theorem}
It turns out that \eqref{defrho} can be quite easily guaranteed.   
\begin{theorem}\label{charinfinity}
Assume that the environment sequence $\xi $ is stationary and ergodic such that \eqref{meanpos} holds. If 
$W$ is not identically zero and 
$\P (\xi _0\ \mbox{not degenerate})>0$, then
$$\limsup_{|t|\to \8 }|\psi (t,\xi )|<1.$$
\end{theorem}

\begin{proof}[Proof of Theorem \ref{Kap}]
Suppose that $W$ is not degenerate and \eqref{meanpos} is satisfied. Then it follows from Theorems \ref{charinfinity} and \ref{integrability} that for almost every $\xi $, $\int _{\R }|\psi '(t,\xi )|\ dt <\8 $. Hence by Lemma \ref{distr},
(iii) in Theorem \ref{Kap} holds. Moreover, $z(\xi )<1$ a.s.

If $W$ is degenerate then it follows from Lemma \ref{nondegiffnondeg} below
that $\P (\xi _0 \ \mbox{is degenerate} )=1$.
\end{proof}
\section{Proof of Theorem \ref{charinfinity}}
We first need some auxiliary results.
\begin{lemma}\label{nondegiffnondeg}
Suppose that $W$ is not identically zero and $W$ is degenerate, i.e. $\Var_{\xi} W=0$. Then 
$\P (\xi _0 \ \mbox{is degenerate}) =1$.
\end{lemma}
\begin{proof}
Taking conditional expectation of both sides of \eqref{Wequation}, we
see that $\E _{\xi}W= \E _{T\xi}W$ and so by ergodicity, $\E _{\xi}W$ is a strictly positive constant, call it $\gamma$. Moreover, due to \eqref{Wequation},
\begin{equation}\label{var}
\Var_{\xi} W= \frac{1}{m_0}\Var_{T\xi} W+\frac{\gamma^2}{m_0^{2}}\Var_{\xi} Z_{1 },\end{equation}
(which holds also in the case when one of the terms is infinite). Suppose that $\Var_{\xi} W=0$. Then iterating \eqref{var}, we have that $\Var_{T^i\xi} Z_{1 }=0$ for all $i\in\N$, which is not possible. Indeed, if 
 $\P (\xi _0 \ \mbox{is not degenerate}) >0$ then by Birkhoff's ergodic theorem for a.e. $\xi $ there is $i$ such that $(T^i\xi )_0=\xi _i$ is not degenerate.
\end{proof}

\begin{lemma}
\label{lem:boudedness of characteristic function}
Assume that $W$ is not identically zero and that $\P (\xi _0\ \mbox{not degenerate})>0$ 
and \eqref{meanpos} holds. Then there is a measurable function $\xi\mapsto (N(\xi),c(\xi))\in\N\times [0,1]$ such that for a.e. $\xi$, $c(\xi)>0$ and
$$|\psi(t,\xi)|\le1-c(\xi)t^2, \quad\text{for } 0\le t\le \tfrac1{2N(\xi)}$$
\end{lemma}
\begin{proof}
First let us observe that  for $|t|\le 1/2$ we have
\begin{align*}
\frac12(e^{it}+e^{-it})=\cos(t)\le 1-t^2/4.
\end{align*}
In particular, let $W'$ be a random variable such that under $\P_{\xi}$, $W$ and $W'$ are i.i.d, then we get
\begin{align*}
|\psi(t,\xi)|^2&=\E_{\xi}\big[e^{itW}\big]\E_{\xi}\big[e^{-itW'}\big]=\E_{\xi}\big[e^{itW-itW'}\big]=\frac12\E_{\xi}\big[e^{it(W-W')}+e^{-it(W-W')}\big]\\
&\le \E_{\xi}\Big[(1-\frac14(t(W-W'))^2)\1_{[|t(W-W')|\le 1/2]}\Big]+\P_{\xi}\big(|t(W-W')|> 1/2\big)\\
&=1-\frac{t^2}4\E_{\xi}\big[(W-W')^2\1_{[|t(W-W')|\le 1/2]}|\big].
\end{align*}
Next, for any natural number $n$ the mapping $\xi\mapsto \E_{\xi}[(W-W')^2\1_{[|W-W'|\le n]}]$ is measurable as well as $\xi\mapsto \E_{\xi}\big[(W-W')^2\big]=2\Var_{\xi} W\in[0,\infty]$. In particular,
$$c(\xi):=\frac1{8}\min(1,\Var_{\xi} W)$$ and $$N(\xi):=\inf\Big\{n:\E_{\xi}\big[(W-W')^2\1_{[|W-W'|\le n]}\big]\ge 8c(\xi)\Big\}$$are also measurable. Next, due to our assumption $\P (\xi _0\ \mbox{not degenerate})>0$ and Lemma \ref{nondegiffnondeg}, we have $c(\xi)>0$ for a.e. $\xi$. Finally,  for $t\le \frac1{2N(\xi)}$ we have 
$$\E_{\xi}\big[(W-W')^2)\1_{[|t(W-W')|\le 1/2]}\big]\ge \E_{\xi}\big[(W-W')^2)\1_{[|(W-W')|\le N(\xi)]}\big]\ge 8c(\xi ),$$ 
and consequently for such $t$  
\begin{align*}
|\psi(t,\xi)| \le\sqrt{1-2c(\xi)t^2}\le1-c(\xi)t^2.
\end{align*}
\end{proof}

\begin{lemma}\label{characteristic}
Assume that the environment sequence $\xi $ is stationary and ergodic such that \eqref{meanpos} holds. If $W$ is not degenerate then for any $0<\beta <1$ there are constants $ c>0$ and  $t_0\leq 1$ such that for a.e. $\xi $ there is a sequence of natural numbers $n_i$ such that
\begin{equation}\label{seq}
(1-\beta )i\leq n_i\leq i, \quad \mbox{for}\ i\geq i_0
\end{equation}  
and
\begin{equation*}
|\psi (t,T^{n_i}\xi )|\leq 1-ct^2, \quad{for}\ 0\leq t \leq t_0. 
\end{equation*}
\end{lemma}

\begin{proof}
Given $\beta\in (0,1)$ there are $c>0, N\in\N$ such that probability of the set
$$
S=\{\xi:  c(\xi)\ge c, \,\, N(\xi)\leq N \}$$
is larger than $1-\beta $. By the ergodic theorem, we have for sufficiently large $n$
$$
\sum _{j=1}^n\1 _S(T^j\xi )\geq (1-\beta )n.$$
{Therefore, for every large enough} $i$ we can find $(1-\beta)i\le n_i\le i$ such that $T^{n_i}\xi \in S$. 
In view of Lemma \ref{lem:boudedness of characteristic function}, for $t\leq t_0:=\tfrac1{2N}$ we have
$$
|\psi (t,T^{n_i}\xi )|\leq 1-ct^2.$$
\end{proof}
\begin{proof}[Proof of Theorem \ref{charinfinity}]
We write, using \eqref{recpsi},
\begin{align*}
\psi (t,\xi )&=\E_{\xi}\bigg[\psi  \Big (\frac{t}{M_n},T^n\xi \Big )^{Z_{n}}\bigg] \\
&=\E_{\xi}\bigg[\psi  \Big (\frac{t}{M_n},T^n\xi \Big )^{Z_{n}}\1 _{[W=0]}\bigg]
+\E_{\xi}\bigg[\psi  \Big (\frac{t}{M_n},T^n\xi \Big )^{Z_{n}}\1 _{[W>0]}\bigg] 
\end{align*}
and the absolute value of the first term above is bounded by $\P _{\xi} (W=0)<1$. 

It remains to estimate the second term. 
By the ergodic theorem we get that for every $0<\eps <\mu 
$ and almost every $\xi $
\begin{equation}\label{means}
e^{i(\mu - \eps)}\leq M_i\leq e^{i(\mu +\eps)}
\end{equation}
for sufficiently large $i$.
Fix $\beta , \eps >0$ such that  $\g = \beta \mu +2\eps <(\mu -\eps)\slash 2 $. 
Then $M_{n_i}\to \8 $ and for sufficiently large $i$ 
we have
\begin{equation}\label{quotientup}
\frac{M_{n _{i+1}}}{M_{n_i}}\leq  e^{\gamma n_i\slash (1-\beta)}=:\a _{n_i}. 
\end{equation}   
Indeed, in view of \eqref{means} and \eqref{seq}
\begin{equation*}
\frac{M_{n_{i+1}}}{M_{n_i}}\leq e^{(\mu +\eps )i} e^{-(\mu -\eps )(1-\beta )i }\leq e^{\gamma i}\leq e^{\gamma n_i\slash (1-\beta)} .
\end{equation*}

For large enough $i_0=i_0(\xi)$ the intervals  $[\a _{n_i}^{-1}t_0M_{n_{i+1}}, t_0M_{n_{i+1}}]$, for $i\ge i_0$, cover $[ t_0 M_{n_{i_0}},\8 )$. 
 Indeed, given $x\geq M_{n_{i_0}}t_0$, let 
 $$
 M_{n_i}=\max \{ M_{n_k}: t_0M_{n_k}\leq x , k\geq i_0\}.$$
 Moreover, we may assume that $i$ is maximal with that property. Then
 $x< t_0M_{n_{i+1}}$ and $x\in [\a _{n_i}^{-1}t_0M_{n_{i+1}}, t_0 M_{n_{i+1}}]$. 

Further, for
 $\a ^{-1} _{n_i}t_0M_{n_{i+1}}\leq |t|\leq t_0M_{n_{i+1}}$ we have
\begin{align*}
\bigg|\E_{\xi}\bigg[\psi  \Big (\frac{t}{M_{n_{i+1}}},T^{n_{i+1}}\xi \Big )^{Z_{n_{i+1}}}\1 _{[W>0]}\bigg]\bigg |
&\leq \E_{\xi}  \bigg[\bigg |\psi  \Big (\frac{t}{M_{n_{i+1}}},T^{n_{i+1}}\xi \Big )\bigg |^{Z_{n_{i+1}}}\1 _{[W>0]}\bigg]
\\
&\leq \E_{\xi} \Big[\big (1-c(|t|M^{-1}_{n_{i+1}})^2\big )^{Z_{n_{i+1}}}\1 _{[W>0]}\Big]
\\ 
&\leq \E_{\xi} \Big[\big (1-ct_0^2 e^{-2\g n_i\slash (1-\beta )}\big )^{Z_{n_{i+1}}}\1 _{[W>0]}\Big]
\end{align*}
For the ease of notation, let $\theta _{n_i}=ct_0^2e^{-2\g n_i\slash (1-\beta )}$ and $W_{n_{i+1}}=\frac{Z_{n_{i+1}}}{M_{n_{i+1}}}$. Then since $1-\theta _{n_i}\leq e^{-\theta _{n_i}}$ we have
\begin{align*}
\Big|\E_{\xi}\Big[\psi  \Big (\frac{t}{M_{n_{i+1}}},T^{n_{i+1}}\xi \Big )^{Z_{n_{i+1}}}\1 _{[W>0]}\Big]\Big |
&\leq 
\E_{\xi} \Big [e^{-\theta _{n_i}Z_{n_{i+1}}}\1 _{[W>0]}\Big]
\\
&=\E_{\xi} \Big[e^{-\theta _{n_i}M_{n_{i+1}}W_{n_{i+1}}}\1 _{[W>0]}\Big].
\end{align*}
Since, for large enough $i$, we have
$$
M_{n_{i+1}}\geq e^{(\mu -\eps )n_{i+1}}\geq e^{(\mu -\eps)(i+1)\slash (1-\beta )}$$
and 
$$
2\g n_i\slash (1-\beta )\leq (\mu -\eps) i\slash (1-\beta )$$
we conclude
$$ 
\lim _{n_i\to \8 }\theta _{n_i}M_{n_{i+1}}=\infty.$$
In particular, by the dominated convergence theorem we infer
$$
\limsup _{i\to \8} \E_{\xi} \Big [e^{-\theta _{n_i}M_{n_{i+1}}W_{n_{i+1}}}\1 _{[W>0]}\Big ]=  \E_{\xi} \Big [ \limsup _{i\to \8} e^{-\theta _{n_i}M_{n_{i+1}}W_{n_{i+1}}}\1 _{[W>0]}\Big ]=0.$$
  
\end{proof}

\section{Integrability of $\psi '$}
In this section we prove Theorem \ref{integrability}. To this end, we need the following auxiliary result.
\begin{lemma}\label{functionh}
Fix $\xi _0$ such that $\xi _0(0)<1$. Let $f=f_{\xi _0}$ and 
$$
h(r)=\frac{1-r}{1-f(r)}f'(r),\quad  0\leq r<1.$$ Then
$$
h(r)\leq \frac{1}{1+f'(1)^{-1}\sum _{k=2}^{\8 }\xi _0(k)(1-r)^{k-1}}\le 1.$$
\end{lemma}
\begin{remark}
	The idea to consider the function $h$ is borrowed from \cite{bansaye:boinghoff:2014}.
\end{remark}
\begin{proof}
First, let us observe that for any  $0<r<1$ we have  
$$
1=f(1)\geq \sum _{k=0}^{\8 }\frac{f^{(k)}(r)}{k!}(1-r)^k.$$
Indeed, by applying Taylor's theorem at  $r$ we get that for any natural number $N$
$$
1=f(1)=\sum _{k=0}^{N }\frac{f^{(k)}(r)}{k!}(1-r)^k +\frac{f^{(N+1)}(s )}{(N+1)!}(1-r)^{N+1}, $$
for some $1-r<s <1$
and since all the derivatives are positive we can take the limit for $N\to \8$ and obtain the desired inequality.
Next, we conclude that
$$
h(r)\leq\frac{(1-r)f'(r )}{f'(r)(1-r)+R_1}\leq 1,$$
where the reminder $R_1$ is given by
$$
R_1=\sum _{k=2}^{\8 }\frac{f^{(k)}(r)}{k!}(1-r)^k.$$
From the fact that $\xi _0(0)<1 $ a.s. we infer
$$
f'(r)=\sum _{m=1}^{\8 }m \xi _0(m)r^{m-1}>0,\quad a.s.$$ 
 and
$$
h(r)\leq \frac{1}{1+f'(r)^{-1}(1-r)^{-1}R_1}.$$
Since all derivatives of $f$ are nonnegative and so nondecreasing, we conclude
\begin{align*}
f'(r)\leq f'(1)
\end{align*}
 and $$\xi _0(k)k!= f^{(k)}(0)\le f^{(k)}(r).$$
In particular, we can estimate the reminder from below
$$
R_1\geq \sum _{k=2}^{\8 }\xi _0(k)(1-r)^k$$
and for  $0\le r\le 1$ we have
\begin{align*}
h(r)&\leq \frac{1}{1+f'(1)^{-1}\sum _{k=2}^{\8 }\xi _0(k)(1-r)^{k-1}}.
\end{align*}
\end{proof}
Now we are ready to prove Theorem \ref{integrability}.
\begin{proof}[Proof of Theorem \ref{integrability} ]

Given $\beta >0$, choose $0<d(\beta )<1$ such that for $S_1=\{ \xi : \rho (\xi )<d (\beta)\}$ we have $\P (S_1)>1-\beta $. By 
the ergodic theorem we conclude that for a.e. $\xi $ and sufficiently large $n$ 
$$
\sum _{j=1}^{n }\1 _{S_1}(T^j\xi )> (1-\beta )n.$$
Therefore, we may choose a sequence $n_i\to \8 $ such that 
\begin{equation}\label{subseq}
 (1-\beta )i <n_i\leq i\end{equation} and 
$$\rho (T^{n_i}\xi )<d (\beta ) $$
for sufficiently large $i$.

Notice that due to $\mu >0$,
\begin{equation}\label{nontriv1}
\P (\xi _0(0)+\xi _0(1)<1)>0.
\end{equation} 
Moreover, our assumptions imply that 
\begin{equation}\label{nontriv}
\P (\xi _0(0)<1)=1.
\end{equation}
Indeed, let $\tilde S=\{ \xi : \xi _0(0)=1\} $ and $\P (\tilde S)>0$. Then, by the PoincarÃ© recurrence theorem, for a.e. $\xi $ there is $n$ such that $T^n\xi \in \tilde S$ i.e. $\xi _0(0)=1$ and so $Z_{n+1}=0$ hence $W=0$ a.s.,
which contradicts \eqref{defrho}.  
Let us introduce
\begin{equation}
b(\xi )=\frac{1}{1+f'_0(1)^{-1}\sum _{k=2}^{\8 }\xi _0(k)(1-\rho (T\xi ))^{k-1}}.
\end{equation}
In view of \eqref{means} and \eqref{nontriv1}, there are 
 $0<\eta <1$ and $\chi >0$ such that for $S=\{ \xi : b (\xi )<\eta \}$ we have $\P (S)>\chi $.
Take $0<\eps <\mu \slash 4$ and $\beta $ such that
\begin{equation}\label{eta}
\chi |\log \eta |>4\gamma (1-\beta )^{-1},
\end{equation}
where $\gamma =\beta \mu +2\eps$. Then $M_{n_i}\to \8 $ and for sufficiently large $i$ 
\eqref{quotientup} is satisfied
and
\begin{equation}\label{quotientdown}
\frac{M_{n_i}}{M_j}\geq 1,\quad \mbox{provided}\ \frac{\chi n_i}{4}\leq j \leq \frac{n_i}{2}.  
\end{equation}   
Indeed, in view of \eqref{means} and \eqref{subseq},
\begin{align*}
M_{n_i}M_j^{-1}\geq &e^{(\mu - \eps )n_i}e^{-(\mu + \eps )j}=e^{\mu (n_i-j)- \eps (n_i+j)}\\
&\geq e^{\mu n_i\slash 2 - 2\eps n_i}\geq 1,
\end{align*}
by the choice of $\eps$. 
As before, in view of \eqref{quotientdown} there is $i(\xi)$ such that for $i\geq i(\xi)$ the intervals
$[M_{n_i}, M_{n_i}\alpha _{n_i}]$ cover $[M_{n_{i(\xi)}},\8 )$. Therefore,
\begin{align*}
\int _{|t|\geq M_{n_{i(\xi)}} }|\psi '(t,\xi )|\ dt 
&\leq \sum _{i\geq i(\xi )}\int _{M_{n_i}\leq |t|\leq \alpha  _{n_i }M_{n_i}}|\psi '(t,\xi )|\ dt\\
&=\sum _{i\geq i(\xi )}\int _{M_{n_i}\leq |t|\leq \a _{n_i}M_{n_i}} |F'_{n_i} (\psi (t\slash M_{n_i} ,T^{n_i}\xi )\psi '(t\slash M_{n_i},T^{n_i}\xi )|M_{n_i}^{-1}\ dt\\
&=\sum _{i\geq i(\xi )}\int _{1\leq |y|\leq \a _{n_i}} |F'_{n_i} (\psi (y ,T^{n_i}\xi ))||\psi '(y,T^{n_i}\xi )|\ dy
\end{align*}
since $\psi (t,\xi )=F_{n_i}(\psi (t\slash M_{n_i},T^{n_i}\xi ))$.
Moreover, 
$$
|\psi '(y, T^{n_i}\xi )|\leq \E _{T\xi}W\leq 1.$$
For any  $n$ and any complex number $z $ in the unit disk we have (to avoid a double subscript we will write  $ f_{k}$ instead of $f_{\xi _{k}}$):
\begin{align*}
|F'_n(z, \xi)|&=\prod _{j=0}^{n-1}|f'_j (f_{j+1}\circ ...\circ f_{n-1}(z)|
\leq \prod _{j=0}^{n-1}f'_j (|f_{j+1}\circ ...\circ f_{n-1}(z)|)\\
&=\prod _{j=0}^{n-1}\frac{1-|f_{j+1}\circ ...\circ f_{n-1}(z)|}{1-|f_{j}\circ ...\circ f_{n-1}(z)|}f'_j (|f_{j+1}\circ ...\circ f_{n-1}(z)|)
\times \frac{1-|f_{0}\circ ...\circ f_{n-1}(z)|}{1-|z|}\\
&\leq\prod _{j=0}^{n-1}\frac{1-|f_{j+1}\circ ...\circ f_{n-1}(z)|}{1-f_j(|f_{j+1}\circ ...\circ f_{n-1}(z)|)}f'_j (|f_{j+1}\circ ...\circ f_{n-1}(z)|)
\times \frac{1-|f_{1}\circ ...\circ f_{n-1}(z)|}{1-|z|},
\end{align*} 
since  $|f_i^{(k)}(z)|\le f_i^{(k)}(|z|)$ for any $k,\ i\ge0$ and any complex $|z|\le1$.
We intend to prove that for sufficiently large $i$
\begin{equation}
\label{eq:decay of the product}
\prod _{j=0}^{{n_i}-1}\frac{1-|f_{j+1}\circ ...\circ f_{{n_i}-1}(z)|}{1-f_j(|f_{j+1}\circ ...\circ f_{{n_i}-1}(z)|)}f'_j (|f_{j+1}\circ ...\circ f_{{n_i}-1}(z)|)\leq \eta ^{{\chi n_i}\slash 4-2}.
\end{equation}
uniformly for $s=\psi (y, T^{n_i}\xi )$ and $|y|\geq 1$.
Hence,
\begin{equation*}
\int _{1\leq |y|\leq \a _{n_i}} |F'_{n_i} (\psi (y ,T^{n_i}\xi ))||\psi '(y,T^{n_i}\xi )|\ dt \leq (1-d (\beta ) )^{-1}
\eta ^{{\chi n_i}\slash 4-2}e^{\gamma n_i\slash (1-\beta )}.
\end{equation*}
Then in view of \eqref{eta}, $\eta ^{{\chi n_i}\slash 4}e^{{n_i}\gamma \slash (1-\beta )}$ decays exponentially and so $\psi '(\cdot ,\xi )$ is integrable.

We return now to show that the inequality \eqref{eq:decay of the product}  holds. To this end, we first prove that for almost every $\xi $, sufficiently large $i$,  $\frac{{\chi n_i}}{4}\leq j \leq \frac{{n_i}}{2}$ and $|y|\geq 1$, we have
$$
  |f_{j+1}\circ ...\circ f_{{n_i}-1}(\psi (y,T^{n_i}\xi ))|\leq \rho (T^{j+1}\xi ).$$
Indeed, 
\begin{equation*}
f_{j}\circ ...\circ f_{{n_i}-1}(\psi (y,T^{n_i}\xi ))=\psi (M_{n_i}M_j^{-1}y,T^j\xi ).
\end{equation*}
So by \eqref{quotientdown}, for $|y|\geq 1$
\begin{equation*}
|\psi (M_{n_i}M_j^{-1}y,T^j\xi )|\leq \sup _{|y|\geq 1}|\psi (y,T^j\xi )|=\rho (T^j\xi )|.
\end{equation*}
For $r\in [0,1]$ consider
$$
h_j(r)=\frac{1-r}{1-f_j(r)}f'_j(r).$$
By Lemma \ref{functionh}, for $0\leq r\leq 1$
$$
h_j(r)\leq \frac{1}{1+f'_j(1)^{-1}\sum _{k=2}^{\8 }\xi _j(k)(1-r)^{k-1}}\leq 1.$$
Let $r=|f_{j+1}\circ ...\circ f_{{n_i}-1}(\psi (y,T^{n_i}\xi ))|$. For ${\chi n_i}\slash 4\leq j\leq {n_i}\slash 2$ and $i$ sufficiently large we have
\begin{align*}
h_j(r)&\leq \frac{1}{1+f'_j(1)^{-1}\sum _{k=2}^{\8 }\xi _j(k)(1-r)^{k-1}}\\
&\leq \frac{1}{1+f'_j(1)^{-1}\sum _{k=2}^{\8 }\xi _j(k)(1-\rho (T^{j+1}\xi ))^{k-1}}\\
&=b(T^{j}\xi ).
\end{align*}
Hence
$$
\prod _{j=0}^{{n_i}-1}\frac{1-|f_{j+1}\circ ...\circ f_{{n_i}-1}(z)|}{1-f_j(|f_{j+1}\circ ...\circ f_{{n_i}-1}(z)|)}f'_j (|f_{j+1}\circ ...\circ f_{{n_i}-1}(z)|)\leq 
\prod _{j=\lceil {\chi n_i}\slash 4 \rceil }^{\lfloor {n_i}\slash 2 \rfloor }b(T^{j}\xi )$$
for a.e. $\xi $ and sufficiently large $i$.
Since $n_i\to \infty $ 
\begin{equation}
\frac{1}{\lfloor {n_i}\slash 2\rfloor} \sum _{j=0}^{\lfloor {n_i}\slash 2\rfloor}\1 _S(T^j\xi )\to \P (S)>\chi , \quad \mbox{a.s.}. 
\end{equation}
So there is $N(\xi )$ such that for $n_i\geq N(\xi )$, $T^j\xi \in S$ at least $\chi \lfloor {n_i}\slash 2\rfloor =\chi n_i\slash 2-1$ times, that is
$$
b(T^j\xi )<\eta $$ at least ${\chi n_i}\slash 4-2$ times for $j>{\chi n_i}\slash 4$. 
Finally,
$$
\prod _{j=1}^{n_i}h_j(r)\leq \eta ^{{\chi n_i}\slash 4 -2}.$$  
\end{proof}

\noindent
{\bf Acknowledgements}\\
The research was partly carried out during visits of E.\;Damek and K.\;Kolesko to Tech\-nische Universit\"at M\"unchen and during visits of
N.\;Gantert to the University of Wroc\l aw.
Grateful acknowledgement is made for hospitality to both universities. The first author was partly supported by NCN Grant DEC-2014/15/B/ST1/00060, the third author was partly supported by NCN Grant DEC-2014/14/E/ST1/00588 and DFG Grant ME 3625/3-1.
We thank Alexander Iksanov for pointing us to \cite{kaplan:1974}.

\providecommand{\bysame}{\leavevmode\hbox to3em{\hrulefill}\thinspace}
\providecommand{\MR}{\relax\ifhmode\unskip\space\fi MR }
\providecommand{\MRhref}[2]{%
  \href{http://www.ams.org/mathscinet-getitem?mr=#1}{#2}
}
\providecommand{\href}[2]{#2}

\end{document}